\documentclass[11pt,leqno]{amsart}
\usepackage{amssymb}
\usepackage[all]{xy}
\usepackage{mathrsfs}
\usepackage{cancel}
\newtheorem{theorem}{Theorem}[section]
\newtheorem{lemma}[theorem]{Lemma}

\newtheorem{proposition}[theorem]{Proposition}
\newtheorem{definition}[theorem]{Definition}
\newtheorem{remark}[theorem]{Remark}

\begin{document}
\title{On the local solvability of a class of degenerate second order operators with complex coefficients}
\author{Serena Federico}
\author{Alberto Parmeggiani}
\address{Department of Mathematics, University of Bologna,
        Piazza di Porta S.Do\-na\-to 5, 40126 Bologna, ITALY}
\email{serena.federico2@unibo.it}
\email{alberto.parmeggiani@unibo.it}
\date{}

\begin{abstract}
We study the local solvability of a class of operators with multiple characteristics. The class considered here
complements and extends the one studied in \cite{FP}, in that in this paper we consider some cases of  operators with complex coefficients
that were not present in \cite{FP}.
The class of operators considered here ideally encompasses classes of degenerate parabolic and Schr\"odinger 
type operators. 
We will give local solvability theorems. In general, one has $L^2$ local solvability, but also cases of local solvability with
better Sobolev regularity will be presented.
\end{abstract}
\maketitle
\pagestyle{myheadings}
\markboth{S. FEDERICO, A. PARMEGGIANI}{
On the local solvability of a class $\ldots$}

\section{Introduction}\label{secIntrod}
\renewcommand{\theequation}{\thesection.\arabic{equation}}
\setcounter{equation}{0}
In this paper we study the local solvability of operators $P$ defined on an open set $\Omega\subset\mathbb{R}^n$, of the form
\begin{equation}
P=\sum_{j=1}^NX_j^*fX_j+X_{N+1}+iX_0+a_0,
\label{eqP1}\end{equation}
and of the form
\begin{equation}
P=\sum_{j=1}^NX_j^*f_jX_j+X_{N+1}+a_0,
\label{eqP2}\end{equation}
where $a_0$ is a smooth complex-valued function and
\begin{itemize}
\item the $X_j=X_j(x,D)$, $0\leq j\leq N+1$, are \textit{homogeneous
first order partial differential operators} (i.e. with no lower order terms; in other words, the $iX_j$ are vector fields) 
with smooth coefficients in $\Omega$, such that the symbols of $X_{N+1}$ and $X_0$ are always \textit{real} and the symbols of
$X_1,\ldots,X_N$ are \textit{real} when $P$ is of the form (\ref{eqP1}), and \textit{complex} when $P$ is of the form (\ref{eqP2});
\item the $f_j\in C^\infty(\Omega;\mathbb{R})$ for $1\leq j\leq N$, and 
\item $f\colon\Omega\longrightarrow\mathbb{R}$ is a smooth function with $S:=f^{-1}(0)\not=\emptyset$ and $df\bigl|_S\not=0$.
\end{itemize}

The operators of the form (\ref{eqP1}) will be called of \textit{mixed-type} (because of the presence of the complex coefficients operator $X_{N+1}+iX_0$, with $X_0\not=0$)
and those of the form (\ref{eqP2}) will be called of \textit{Schr\"odinger-type} (because of the presence of the real coefficients
operator $X_{N+1}$ only, $X_0$ being identically zero).

The class of operators of the form (\ref{eqP1}) and (\ref{eqP2}) enlarges and complements that studied in \cite{FP} 
(in turn, a generalization of the class introduced by Colombini, Cordaro and Pernazza in \cite{CCP}) of operators of the form
$$P=\sum_{j=1}^NX_j^*fX_j+iX_0+a_0,$$
which has as an important ancestor the Kannai operator (and the class considered by
Beals and Fefferman in \cite{BF}). In fact, as already explained earlier, here we allow cases
in which the $X_0,X_1,\ldots,X_N, X_{N+1}$ have a \textit{real} symbol but with $X_0\not=0$ (the \textit{mixed-type} case of Section \ref{sec2}),
and cases  in which the $X_1,\ldots,X_N$ are allowed to have a \textit{complex} symbol but with $X_{N+1}\not=0$ and
$X_0=0$ (the \textit{Schr\"odinger-type} case of Section \ref{sec3}).

Note that in \cite{FP} we did allow a complex case in which the $X_1,\ldots,X_N$ were complex but there we had $X_0\not=0$
and $X_{N+1}=0$. 

Our main motivation in studying such a class of degenerate differential operators is to push the frontier for the solvability in presence of
multiple characteristics. 
Besides the papers \cite{BF}, \cite{CCP}, \cite{FP} and \cite{F} (in which a case with non-smooth coefficients is studied), and the book \cite{L} (where one
can find an updated account of the solvability issue under the ($\Psi$) condition of Nirenberg and Treves, problem solved by Dencker in \cite{D}),
we wish to recall a number of works related to the local solvability of operators with multiple characteristics, 
such as \cite{Pop}, \cite{MU, MU1}, \cite{MR}, \cite{T, T2}, \cite{PeR}, \cite{M}, \cite{K}, \cite{PP}, and \cite{D1, D2} (see also \cite{P} 
and references therein). In particular, among them we wish to single out the recent paper \cite{D2} by Dencker in which he introduces the class of sub-principal type 
operators (whose characteristics are involutive) for which he gave necessary conditions for local solvability, and the paper \cite{PP} by Parenti and Parmeggiani 
(see also \cite{P}) in which they obtain semiglobal solvability results (with a loss of many derivatives) for operators with transversal multiple symplectic characteristics. .  
In the case of the class of operators we consider in this paper, we aim to give sufficient conditions for local solvability in presence of an 
interplay of different kinds of degeneracies, namely that coming from the change of sign of $f$, or $f_j$, in (\ref{eqP1}), and (\ref{eqP2}), 
and that coming from the system of vector fields $(iX_0,\ldots,iX_N)$.
This class is all the more interesting in that it contains operators whose adjoint is \textit{not} hypoelliptic. 

In \cite{FP} we used a "positive commutator method" that, starting from estimating $|\!|P^*u|\!|_0^2$, 
could make use of fundamental lower-bound estimates 
(the G\aa rding, the Melin, the Fefferman-Phong, and the Rothschild-Stein subelliptic estimates for H\"ormander's sums of squares).
In the present case, such a method cannot be used (as one can easily see, for instance, from the Schr\"odinger operator 
$P=D_t+A$, since when estimating $|\!|P^*u|\!|_0^2=|\!|A^*u|\!|_0^2
+|\!|D_tu|\!|_0^2+2\,\mathsf{Re}\,(A^*u,D_tu)$ one is not able to directly extract any extra information coming from the term 
$2\,\mathsf{Re}\,(A^*u,D_tu)$ as one could in \cite{FP}). We will have to make a Carleman estimate straight from the beginning.
In the \textit{mixed-type} case (i.e. $P$ of the form (\ref{eqP1})),
we shall however be once more in a position to exploit the above lower-bound estimates to go, in some cases, beyond the \textit{$L^2$ to $L^2$ local
solvability}, and get a better \textit{$H^{-s}$ to $L^2$ local solvability} (see Definition \ref{terminology} below), with $s=-1/2$ or $s=-1$, or $s=-1/r$ ($r\geq 3$).
In the \textit{Schr\"odinger-type} case, we will not be able to exploit the above lower-bound estimates and 
the Carleman estimate will grant $L^2$ local solvability results under the assumption
that the system of complex operators $X_1,\ldots,X_N$ admits, locally near each $x_0\in\Omega$, a \textit{real} smooth first integral $g$ (i.e., such that
$dg(X_j)=dg(\mathsf{Re}\,X_j)+idg(\mathsf{Im}\,X_j)=0$) near $x_0$, $1\leq j\leq N$, such
that $X_0g\not=0$ near $x_0$.  

Recall the following the terminology introduced in \cite{FP}.
\begin{definition}\label{terminology}
Given $s,s'\in\mathbb{R}$ we say that we have
{\rm $H^s$ to $H^{s'}$ local solvability} if for any given $x_0\in\Omega$ there is a compact $K\subset\Omega$ with 
$x_0\in\mathring{K}$ (the interior of $K$) such that for all $v\in H^s_{\mathrm{loc}}(\Omega)$ there exists 
$u\in H^{s'}_{\mathrm{loc}}(\Omega)$ with $Pu=v$ in $\mathring{K}$. We will call the number $s'-s$ the {\rm gain
of smoothness} of the solution.
\end{definition}

\begin{remark}
It is important to remark once more that the class we consider here, as well as that considered in \cite{FP}, contains operators that are
{\rm not} adjoints of hypoelliptic operators (see \cite{P}, Example 3.7).
\end{remark}


We next establish some notation that will be used throughout the paper.

In general, for a differential operator with complex coefficients of the form
$X(x,D)=\langle\zeta(x),D\rangle,$ where $D=(D_1,\ldots,D_n)$, $D_k=-i\partial_k$, and
$\zeta\in C^\infty(\Omega;\mathbb{C}^n),$ we have
\begin{equation}
X(x,D)^*=\bar{X}(x,D)+d_{\bar{X}}(x),
\label{eqAdj}\end{equation} 
where 
$$\bar{X}(x,D)=\langle\overline{\zeta(x)},D\rangle,\quad\text{\rm and}\quad d_{\bar{X}}(x)=\sum_{k=1}^nD_k\overline{\zeta_k(x)}=-\overline{d_X(x)}.$$
Therefore, in general for the formal adjoints of the $X_j(x,D)$ we have that $X_j(x,D)^*=\bar{X}_j(x,D)+d_{\bar{X}_j}(x)$ 
and, since $X_0(x,\xi)$ and $X_{N+1}(x,\xi)$ are real,
$$d_{X_0}(x)=d_{\bar{X}_0}(x)=-\overline{d_{X_0}(x)},\quad d_{X_{N+1}}(x)=d_{\bar{X}_{N+1}}(x)=-\overline{d_{X_{N+1}}(x)},$$
so that, in particular, $d_{X_0},$ $d_{X_{N+1}}$ are \textit{purely imaginary}. 

In the case of $P$ of the form (\ref{eqP1}), we put 
$$\Sigma_j:=\{(x,\xi)\in T^*\Omega\setminus 0;\,\,X_j(x,\xi)=0\},\quad 0\leq j\leq N,$$
\begin{equation}
\Sigma:=\bigcap_{j=0}^N\Sigma_j\subset T^*\Omega\setminus 0,
\label{eqCharX}\end{equation}
and call $\Sigma$ the \textit{characteristic set of the system $(X_0,X_1,\ldots,X_N)$}. The kind of degeneracy of an
operator $P$ of the form (\ref{eqP1}) therefore comes from the interplay of the location of $\pi^{-1}(S)$ with respect to $\Sigma$ (here $\pi\colon T^*\Omega\longrightarrow\Omega$
denotes the canonical projection), that is, from the zero-set of $f$ and the behavior of the family of 
operators $X_j$, $0\leq j\leq N$ near it.

Notice that the set $\Sigma$ will play a role only in the case of mixed-type operators (\ref{eqP1}), and \textit{not} 
in the Schr\"odinger-type case (\ref{eqP2}).

We conclude this introduction by giving the plan of the paper. 
In Section \ref{sec2} we will consider the mixed-type case in which the $X_0, X_1,\ldots,X_N, X_{N+1}$ have a real symbol and $X_0\not=0$,
and show in Theorem \ref{thmMixedType}, under suitable assumptions on the commutators of $X_0$ with the $X_j$, $1\leq j\leq N+1$, and 
assuming control
of the symbol of $\mathsf{Im}\,d_{X_0}\,X_{N+1}$ by $(\sum_{j=0}^NX_j(x,\xi)^2))^{1/2}$, that one has local solvability near $S$ with a better gain
of smoothness.
In Section \ref{secExMType} we shall give  examples of operators of mixed-type (\ref{eqP1}) to which Theorem \ref{thmMixedType}
can be applied, thus showing the different issues of local solvability with different smoothness.
In Section \ref{sec3} we will consider the Schr\"odinger-type
case $X_0=0$ with $X_1,\ldots,X_N$ having a complex symbol and show in Theorem \ref{thmComplex} that one has $L^2$ to $L^2$ local solvability near any given point of $\Omega$.
In the final Section \ref{secExSType} we will give examples of operators of Schr\"odinger type (\ref{eqP2}) to which Theorem \ref{thmComplex} can be applied. 

\section{The mixed-type case}\label{sec2}

We now turn our attention to an operator $P$ of the form (\ref{eqP1}) (mixed-type case), that is
$$P=\sum_{j=1}^NX_j^*fX_j+X_{N+1}+iX_0+a_0,$$
where the symbols of $X_j,$ $0\leq j\leq N+1$ are all \textit{real} on the open
set $\Omega\subset\mathbb{R}^n$, and where $f\in C^\infty(\Omega;\mathbb{R})$ is such that $S:=f^{-1}(0)$ is non-empty and 
$df\bigl|_S\not=0.$ Recall that, writing
$X_j(x,\xi)=\langle\alpha_j(x),\xi\rangle$ for $\alpha_j\in C^\infty(\Omega;\mathbb{R}^n),$ then $d_{X_j}=\sum_{k=1}^nD_k\alpha_{jk}\in C^\infty(\Omega;i\mathbb{R})$.

Note that in this case the subprincipal symbol of $P$ is given by
$$\mathrm{sub}(P)(x,\xi)=X_{N+1}(x,\xi)+iX_0(x,\xi).$$

In order to prove the a priori inequality that ensures the local solvability result we are interested in, one has to control from below in $L^2$
a quadratic form of the kind $(\widehat{P}_{\gamma,\varepsilon}u,u)$, $u\in C_0^\infty$, where, 
for $\gamma>0$ and $\varepsilon\in(0,1]$ suitably fixed constants,
\begin{equation}
\widehat{P}_{\gamma,\varepsilon}=\widehat{P}_{\gamma,\varepsilon}(x,D):=\sum_{j=0}^N\Bigr(X_j^*X_j
-\varepsilon[X_j,X_0]^*[X_j,X_0]\Bigr)+\frac{1}{\gamma}Y,
\label{eqPgamma}\end{equation}
with $Y$ given by
\begin{equation}
Y:=-\frac12\Bigl((\mathsf{Im}\,d_{X_0})X_{N+1}+((\mathsf{Im}\,d_{X_0})X_{N+1})^*\Bigr)=Y^*.
\label{eqY}\end{equation} 
The point is then to give conditions on the system of real vector fields $iX_0,\ldots,iX_{N+1}$ in relation with $S$ 
in order that $\widehat{P}_{\gamma,\varepsilon}$ satisfy the Fefferman-Phong inequality (with $\gamma$ and $\varepsilon$ suitably
chosen). 

In this section we make the following hypotheses:

\begin{itemize}
\item[(HM1)] $iX_0f>0$ on $S$;
\item[(HM2)] For all $x_0\in S$ there exists a compact $K\subset\Omega$, containing $x_0$ in its interior, and a constant $C_K>0$ such that for all $j=1,\ldots,N+1$
$$\{X_j,X_0\}(x,\xi)^2\leq C_K\sum_{j'=0}^NX_{j'}(x,\xi)^2,\,\,\forall (x,\xi)\in  K\times\mathbb{R}^n;$$
\item[(HM3)] For all $x_0\in S$ there exists a compact $K\subset\Omega$, containing $x_0$ in its interior, and a constant $C_K>0$ such that
$$|(\mathsf{Im}\,d_{X_0}(x))X_{N+1}(x,\xi)|\leq C_K\Bigl(\sum_{j=0}^NX_j(x,\xi)^2\Bigr)^{1/2},\,\,\,\,\forall (x,\xi)\in K\times\mathbb{R}^n.$$
\end{itemize}

\begin{definition}[Hypothesis (HM4)]\label{defHM4}
We shall say that hypothesis (HM4) is satisfied at $x_0\in S$ if $\pi^{-1}(x_0)\cap\Sigma\not=\emptyset$ and
$$\mathrm{Tr}^+F(\rho)>0,\,\,\,\,\forall\rho\in\pi^{-1}(x_0)\cap\Sigma,$$
where $\mathrm{Tr}^+F(\rho)$ is the positive trace of the Hamilton map of the principal symbol of $\sum_{j=0}^NX_j^*X_j$ (see \cite{Ho}). 
\end{definition}

\begin{definition}[Hypothesis (HM5)]\label{defHM5}
Let $\mathscr{L}_k(x)$ be the (real) vector space
generated by the vector fields $iX_0,\ldots,iX_N$ 
along with their commutators of length at most $k$ evaluated at the point $x$.
\footnote{We take this opportunity to correct the statements in \cite{FP} (Thm. 9.2) and \cite{P} (Thm. 3.12) in which the same condition (HM5) appears. In both papers, 
$\mathscr{L}_r(x)$ was meant to be defined as in the definition, and it suffices that the maximality condition on the dimension 
be holding at $x_0$ only.}
We shall say that hypothesis (HM5) is satisfied at $x_0\in S$ if $\pi^{-1}(x_0)\cap\Sigma\not=\emptyset$ and
one has the existence of an integer $r\geq 1$ such that 
$$\dim\mathscr{L}_r(x_0)=n.$$
\end{definition}

In the following remarks we explain the connection of hypotheses (HM4) and (HM5) to the Melin and the Rothschild-Stein
lower-bound estimates. Recall that $\Sigma$ is the characteristic set of the operator $\sum_{j=0}^NX_j^*X_j$.

\begin{remark}\label{remHM4}
Condition (HM4) is equivalent to condition (H3) of \cite{FP}. In fact, let $\rho\in\Sigma$ and let $H_{X_j}(\rho)$ be the Hamilton vector fields of the symbols
$X_j(x,\xi)$ at $\rho.$ Define $V(\rho)=\mathrm{Span}\{H_{X_j}(\rho);\,\,j=0,\ldots N\}$ and let $J=J(\rho)\subset\{0,\ldots,N\}$ be a set of indices for which 
$H_{X_j}(\rho)$, $j\in J$, form a basis of $V(\rho).$ If $r=\sharp J$ and if one considers the $r\times r$ matrix
$$M(\rho)=[\{X_j,X_{j'}\}(\rho)]_{j,j'\in J},$$  
then (HM4) is equivalent to requiring
$$\mathrm{rank}\,M(\rho)\geq 2,\,\,\,\,\forall\rho\in\pi^{-1}(x_0)\cap\Sigma.$$
Note also that if condition (HM4) holds at $x_0$ then there exists a sufficiently small open neighborhood $V_{x_0}$ of $x_0$ such that it holds for all $\rho\in
\pi^{-1}(V_{x_0})\cap\Sigma.$

Finally, since the subprincipal symbol of $\sum_{j=0}^NX_j^*X_j$ is identically zero (the symbols $X_j(x,\xi)$ being real)
we have that condition (HM4) is Melin's strong Tr+ condition
$$\mathrm{sub}(\sum_{j=0}^NX_j^*X_j)(\rho)+\mathrm{Tr}^+F(\rho)>0,\,\,\,\,\forall\rho\in\pi^{-1}(x_0)\cap\Sigma,$$
whence (HM4) yields that for a sufficiently small compact $K$ containing $x_0$ in its interior we have the sharp Melin inequality \cite{Ho}
\begin{equation}
(\sum_{j=0}^NX_j^*X_j u,u)=\sum_{j=0}^N|\!|X_ju|\!|_0^2\geq c_K|\!|u|\!|_{1/2}^2-C_K|\!|u|\!|_0^2,\,\,\,\,\forall u\in C_0^\infty(K),
\label{eqMelin}\end{equation}
for $c_K,C_K$ positive constants. 
\end{remark}

\begin{remark}\label{remHM5}
Condition (HM5) yields the Rothschild-Stein sharp subelliptic estimate in a neighborhood $V_{x_0}$ of $x_0$
(see \cite{RS}, and \cite{HN}): 
For any given compact $K\subset V_{x_0}$ there exists $C_K>0$ such that
\begin{equation}
|\!|u|\!|_{1/r}^2\leq C_K\Bigl(\sum_{j=0}^N|\!|X_ju|\!|_0^2+|\!|u|\!|_0^2\Bigr),\,\,\,\,\forall u\in C_0^\infty(K).
\label{eqHRS}\end{equation}
Note that condition (HM4) (via the sharp Melin inequality), yields (\ref{eqHRS}) with $r=2$. Moreover, 
hypothesis (HM4) is symplectically invariant, 
and the sharp Melin inequality holds true for general pseudodifferential operators. 
Note also that for the full microlocal analogue of (\ref{eqHRS}) one needs the full strength of the maximal hypoelliptic 
estimates of \cite{HN} (see also \cite{BCN}).   
\end{remark}

In this section we will show that under hypotheses (HM1) through (HM3) the operator $P$ of the form (\ref{eqP1}) 
is $L^2$ to $L^2$ locally solvable near any given $x_0\in S$ such that $\pi^{-1}(x_0)\cap\Sigma\not=\emptyset$. 
When in addition hypothesis (HM4) holds then $P$ is $H^{-1/2}$ to $L^2$ locally solvable near such an $x_0$, when (HM4) is
replaced by (HM5) then $P$ is $H^{-1/r}$ to $L^2$ locally solvable near such an $x_0$, and finally when
$x_0$ is such that $\pi^{-1}(x_0)\cap\Sigma=\emptyset$ then $P$ is $H^{-1}$ to $L^2$ locally solvable near such an $x_0$. 
This result generalizes the result of \cite{FP} in that, there, only the case $X_{N+1}=0$ was considered. 
As in \cite{FP}, the point here is to obtain an a priori estimate that makes use of the Fefferman-Phong almost-positivity estimates for
the auxiliary operator $\widehat{P}_{\gamma,\varepsilon}$ and the G\aa rding, or the sharp Melin inequality, or the Rothschild-Stein 
subelliptic estimate, depending on the cases, for the operator $\sum_{j=0}^NX_j^*X_j$.
However, the approach of \cite{FP} cannot be directly used in the present case.

We will prove the following theorem.

\begin{theorem}\label{thmMixedType}
Let $P$ be an operator of the form \eqref{eqP1} defined on an open set $\Omega\subseteq\mathbb{R}^n$. If $P$ satisfies hypotheses (HM1), (HM2) and (HM3), then
\begin{itemize}
\item [(i)] for all $x_0\in S$ 
one has that $P$ is $L^2$ to $L^2$ locally solvable at $x_0$;
\item [(ii)] if $x_0\in S$ is such that $\Sigma\cap\pi^{-1}(x_0)\neq\emptyset$  and condition (HM4) is satisfied at $x_0$ then $P$ is $H^{-1/2}$ to $L^2$ locally solvable at $x_0$;
\item[(iii)] if $x_0\in S$ is such that $\Sigma\cap\pi^{-1}(x_0)\neq\emptyset$  and condition (HM5) is satisfied at $x_0$ then $P$ is $H^{-1/r}$ to $L^2$ locally solvable at $x_0$;
\item [(iv)] if $x_0\in S$ is such that $\Sigma\cap\pi^{-1}(x_0)=\emptyset$ then $P$ is $H^{-1}$ to $L^2$ locally solvable at $x_0$.
\end{itemize}
\end{theorem}

We prepare the proof of Theorem \ref{thmMixedType} by establishing the following key estimate.

\begin{proposition}\label{propKey}
There exists a compact $K\subset\Omega$ containing $x_0$ in its interior and with sufficiently small diameter, and constants $c_K, C_K>0$  
such that for all $u\in C_0^\infty(K)$
\begin{equation}
2\,\mathsf{Re}(P^*u,-iX_0u)\geq c_K\sum_{j=0}^N|\!|X_ju|\!|_0^2
+\frac32|\!|X_0u|\!|_0^2-C_K|\!|u|\!|_0^2.
\label{eqFundIneq}\end{equation}
\end{proposition}

\begin{proof}[Proof of Proposition \ref{propKey}]
Let for short $B=-X_0.$
Fix $x_0\in S$ and consider a compact $K\subset\Omega$ containing $x_0$ in its interior. Write
\begin{equation}
\label{GP1}
2\mathsf{Re}(P^*u,iB u)=\sum_{j=1}^N\underset{(\ref{GP1}.1)}{\underbrace{2\mathsf{Re}(X_j^*f X_ju,iB u)}}+\underset{(\ref{GP1}.2)}{\underbrace{2\mathsf{Re}\big((X^*_{N+1}-iX^*_0)u,iB u\big)}}.
\end{equation} 
Observe that, by suitably shrinking $K$ around $x_0$, hypothesis (HM1) yields the existence of a positive constant $c_0$ such that $-iB f=iX_0f\geq c_0>0$ on $K$.
We then work in this new compact that we still denote by $K$ and estimate (\ref{GP1}.1) and (\ref{GP1}.2) separately.
As for (\ref{GP1}.1) we have that for all $0\leq j\leq N$, 
\begin{equation}
\label{GEst1}
2\mathsf{Re}(X_j^* f X_ju,iBu)=2\mathsf{Re}( f X_ju,i[X_j,B] u)+2\mathsf{Re}(f X_ju,iB X_ju)
\end{equation}
$$=2\mathsf{Re}( f X_ju,i[X_j,B] u)+2\mathsf{Im}(f X_ju,B X_ju)$$
$$=2\mathsf{Re}( f X_ju,i[X_j,B] u)+\frac1i\Big((f X_ju,B X_ju)-(B X_ju,f X_ju)\Big)$$
$$=2\mathsf{Re}( f X_ju,i[X_j,B] u)+\frac1i\Big((B^* f X_ju, X_ju)-(B X_ju,f X_ju)\Big)$$
$$=2\mathsf{Re}( f X_ju,i[X_j,B] u)+\frac1i\Big(((B f )X_ju,X_ju)$$
$$+(d_B f X_ju,X_ju)+\cancel{(B X_ju,fX_ju)}-\cancel{(B X_ju,f X_ju)}\Big)$$
$$\geq -|\!|f|\!|_{L^\infty(K)}\Big((|\!|d_B|\!|_{L^\infty(K)}+1)|\!|X_j u|\!|_0^2+|\!|[X_j,X_0] u|\!|_0^2\Big)+c_0|\!|X_ju|\!|_0^2.$$
As for the term in (\ref{GP1}.2), we have
\begin{equation}
\label{GEst2}
2\mathsf{Re}\big((X^*_{N+1}-i X^*_0)u,iB u\big)=2\mathsf{Im}(X^*_{N+1}u,Bu)-2\mathsf{Re}(iX^*_0u,iBu)
\end{equation}
$$\underset{}{=}\frac1i\Big((X^*_{N+1}u,Bu)-(Bu,X^*_{N+1} u)\Big)+2\mathsf{Re}(-X_0^*u,Bu)$$
$$=\frac1i\Big((X^*_{N+1}u,Bu)-([X_{N+1},B] u,u)-( X_{N+1}u,B^*u))\Big)$$
$$\hspace{3.5cm}+2\mathsf{Re}(-X_0u,Bu)-2\mathsf{Re}(d_{X_0}u,Bu)$$
(recalling that $X_{N+1}^*=X_{N+1}+d_{X_{N+1}}$ and $X_0^*=X_0+d_{X_0}$, and that $d_{X_{N+1}}$ and
$d_{X_0}$ are purely imaginary, see (\ref{eqAdj}))
$$\underset{(B=-X_0)}{=}\mathsf{Re}\Biggl(\frac1i\Big(-(d_{X_{N+1}}u,X_0u)+([X_{N+1},X_0] u,u)+( X_{N+1}u,d_{X_0}u))\Big)\Biggr)$$
$$\hspace{3.5cm}+2|\!|X_0u|\!|_0^2+2\mathsf{Re}(d_{X_0}u,X_0 u)$$      
$$\geq -\frac{1}{2\delta_0}|\!|d_{X_{N+1}}|\!|^2_{L^\infty(K)}|\!|u|\!|_0^2-\frac{\delta_0}{2}|\!|X_0u|\!|_0^2-\frac{1}{2\delta_1}|\!|u|\!|_0^2$$
$$-\frac{\delta_1}{2}|\!|[X_{N+1},X_0]u|\!|_0^2-\mathsf{Re}( \mathsf{Im}\,d_{X_0}\,X_{N+1}u,u)$$
$$-\frac{1}{\delta_2}|\!|d_{X_0}|\!|_{L^\infty(K)}^2|\!|u|\!|_0^2-\delta_2|\!|X_0u|\!|_0^2+2|\!|X_0u|\!|_0^2.$$

Using \eqref{GEst1} and \eqref{GEst2} in \eqref{GP1}, and recalling that $B=-X_0$ in \eqref{GEst1} gives
$$2\mathsf{Re}(P^*u,iBu) \geq \sum_{j=1}^N\Big(c_0-|\!|f|\!|_{L^\infty(K)}(|\!|d_{X_0}|\!|_{L^\infty(K)}+1)\Big)|\!|X_j u|\!|_0^2$$
$$-|\!|f|\!|_{L^\infty(K)}\sum_{j=1}^N|\!|[X_j,X_0] u|\!|_0^2 -\frac{\delta_1}{2}|\!|[X_{N+1},X_0]u|\!|_0^2$$
$$-\mathsf{Re}(\mathsf{Im}\,d_{X_0}\,X_{N+1}u,u)+\Big(2-\frac{\delta_0}{2}-\delta_2\Big)|\!|X_0u|\!|_0^2$$
$$-\Big( \frac{1}{2\delta_0}|\!|d_{X_{N+1}}|\!|_{L^\infty(K)}^2+\frac{1}{2\delta_1}+\frac{1}{2\delta_2}|\!|d_{X_0}|\!|_{L^\infty(K)}^2\Big)|\!|u|\!|_0^2.$$
Since $x_0\in S$ and $K$ contains $x_0$ in its interior, we may shrink the compact set $K$ around $x_0$ to a compact set, that we still denote by $K$, in such a way that 
$|\!|f|\!|_{L^\infty(K)}$ is so small that $c_0-|\!|f|\!|_{L^\infty(K)}(|\!|d_{X_0}|\!|_{L^\infty(K)}+1)\geq c_0/2$. 
We may then also pick $\delta_0$ and $ \delta_2$ sufficiently small in order that $2-\frac{\delta_0}{2}-\delta_2\geq 7/4$. 
Therefore, with so chosen $\delta_0$ and $\delta_2$, with
$$C(\delta_1)=\frac{1}{2\delta_0}|\!|d_{X_{N+1}}|\!|_{L^\infty(K)}^2+\frac{1}{2\delta_1}+\frac{1}{2\delta_2}|\!|d_{X_0}|\!|_{L^\infty(K)}^2>0,$$
with $c'_0=\min\{c_0/2, 1/4\}$, and recalling $Y$ given in (\ref{eqY}) we get, with $\gamma_0:=c'_0/3$,
\begin{equation}
\label{Final}
2\mathsf{Re}(P^*u,iBu) \geq \frac{c_0}{2}\sum_{j=1}^N|\!|X_j u|\!|_0^2-|\!|f|\!|_{L^\infty(K)}\sum_{j=0}^N|\!|[X_j,X_0] u|\!|_0^2
\end{equation}
$$-\frac{\delta_1}{2}|\!|[X_{N+1},X_0] u|\!|_0^2-C(\delta_1)|\!|u|\!|_0^2+(Yu,u)+\frac 74|\!|X_0u|\!|_0^2$$
$$\geq \gamma_0\underset{(\ref{Final}.1)}{\underbrace{\Biggl(\sum_{j=0}^N\Big(|\!|X_j u|\!|_0^2-\frac{1}{\gamma_0}|\!|f|\!|_{L^\infty(K)}|\!|[X_j,X_0] u|\!|_0^2\Big)+\frac{1}{\gamma_0}(Yu,u)\Biggr)}}$$
$$+\gamma_0\sum_{j=0}^N|\!|X_j u|\!|_0^2+\gamma_0\underset{(\ref{Final}.2)}{\underbrace{\Big(\sum_{j=0}^N|\!|X_ju|\!|_0^2-\frac{\delta_1}{2\gamma_0}|\!|[X_{N+1},X_0]u|\!|_0^2\Big)}}\hspace{1.5cm}$$
$$\hspace{5.5cm}-C(\delta_1)|\!|u|\!|_0^2+ \frac32|\!|X_0u|\!|_0^2.$$
Note that (\ref{Final}.1) can be written as $(\widehat{P}_{\gamma_0,\epsilon(K)}u,u)$, with (recall \eqref{eqPgamma})
$$\widehat{P}_{\gamma_0,\epsilon(K)}
=\sum_{j=0}^N\Big(X_j^*X_j-\frac{1}{\gamma_0}|\!|f|\!|_{L^\infty(K)}[X_j,X_0]^*[X_j,X_0]\Big)+\frac{1}{\gamma_0}Y$$
where $\varepsilon(K)=|\!|f|\!|_{L^\infty(K)}/\gamma_0$ is a positive constant that shrinks to zero when $K$ is shrunk around $x_0$, 
that is, $\varepsilon(K)\rightarrow 0$ as $K\rightarrow \{x_0\}$.

At this point we need the following crucial lemma.

\begin{lemma}\label{lemmaFP-M-G}
Suppose (HM2) and (HM3) hold. Then we may shrink $K$, keeping $x_0$ in its interior, to a compact, that we keep calling $K$,
such that
$\widehat{P}_{\gamma_0,\varepsilon(K)}$ satisfies the Fefferman-Phong inequality on $C_0^\infty(K)$
$$(\widehat{P}_{\gamma_0,\varepsilon(K)} u,u)\geq -C_1|\!|u|\!|_0^2,\,\,\,\,\forall u\in C_0^\infty(K),$$
for some constant $C_1>0$ (depending on $K$).
\end{lemma}
\begin{proof}[Proof of the lemma]  
The proof is obtained exactly in the same way of Lemma 6.1 of \cite{FP}. We first shrink $K$, keeping $x_0$ in its interior, so that by virtue of 
(HM2) and (HM3) the total symbol of $\widehat{P}_{\gamma_0,\varepsilon(K)}$ is bounded from below by a constant in a neighborhood of 
$K\times\mathbb{R}^n$. One then extends the total symbol of $\widehat{P}_{\gamma_0,\varepsilon(K)}$  
to a symbol in the class $S^2_{1,0}(\mathbb{R}^n\times\mathbb{R}^n)$, which is still 
bounded from below. The resulting operator, which is still a \textit{differential} operator, satisfies
the Fefferman-Phong inequality and coincides with $\widehat{P}_{\gamma_0,\varepsilon(K)}$ on $C_0^\infty(K).$ 
This concludes the proof of the lemma.
\end{proof}

Lemma \ref{lemmaFP-M-G} allows to control the term (\ref{Final}.1).

As regards the term (\ref{Final}.2) we can write it as $(Q_1u,u)$ with
$$Q_1:=\sum_{j=0}^NX_j^*X_j-\frac{\delta_1}{2\gamma_0} [X_{N+1},X_0]^* [X_{N+1},X_0].$$
Performing on $Q_1$ the same procedure we used in Lemma \ref{lemmaFP-M-G}, we may choose $\delta_1>0$
so as that for $Q_1$ the same conclusion of Lemma \ref{lemmaFP-M-G} holds on $C_0^\infty(K)$, where $K$ is the
resulting compact containing $x_0$ in its interior.

Therefore, for all $u\in C_0^\infty(K),$
\begin{equation}
2\mathsf{Re}(P^*u,iBu)\geq \gamma_0\sum_{j=0}^N|\!|X_ju|\!|_0^2
+\frac32|\!|X_0 u|\!|_0^2-C|\!|u|\!|_0^2,
\label{eqCarleman}\end{equation}
with $C$ a positive constant (depending on the compact).
This concludes the proof of the proposition.
\end{proof}

\begin{proof}[Proof of Theorem \ref{thmMixedType}]
It is now an easy matter to prove the theorem. Since
$$2\mathsf{Re}(P^*u,iBu)\leq 
|\!|P^*u|\!|_0^2+|\!|X_0u|\!|^2,$$
we have
$$|\!|P^*u|\!|_0^2\geq \gamma_0\sum_{j=0}^N|\!|X_ju|\!|_0^2
+\frac12|\!|X_0 u|\!|_0^2-C|\!|u|\!|_0^2.$$
Finally, by using the Poincar\'e inequality on $X_0$ (which is nonsingular on $S$), and by possibly shrinking once more the compact $K$ around $x_0$, we may absorb
the negative constant $-C$ in front of the $L^2$-norm and obtain,
with a new suitable positive constant $C$,
\begin{equation}
|\!|P^*u|\!|_0^2\geq \gamma_0\sum_{j=0}^N|\!|X_ju|\!|_0^2
+C|\!|u|\!|_0^2,\,\,\forall u\in C_0^\infty(K),
\label{eqSolv}\end{equation}
which yields the estimate that guarantees the local solvability of $P$ in the sense $H^{-s}$ to $L^2$ with
$s=0$ in case (i) and $s=1$ in case (iv) of the statement of the theorem.

It remains to deal with cases (ii) and (iii) of the statement. As for (ii), we use hypothesis (HM4) to exploit the sharp Melin inequality (\ref{eqMelin}) and,
using (\ref{eqSolv}), to get
\begin{equation}
|\!|P^*u|\!|_0^2\geq \gamma_0\sum_{j=0}^N|\!|X_ju|\!|_0^2+C|\!|u|\!|_0^2\geq C'|\!|u|\!|_{1/2}^2,\,\,\forall u\in C_0^\infty(K),
\label{solv0}\end{equation}
and hence the $H^{-1/2}$ to $L^2$ local solvability of $P$ near $x_0$.

As for (iii) we make use of hypothesis (HM5) that, by the subelliptic estimate (\ref{eqHRS}) for H\"ormander's sums of square of
vector fields and (\ref{eqSolv}), gives
\begin{equation}
|\!|P^*u|\!|_0^2\geq \gamma_0\sum_{j=0}^N|\!|X_ju|\!|_0^2+C|\!|u|\!|_0^2\geq C'|\!|u|\!|_{1/r}^2,\,\,\forall u\in C_0^\infty(K),
\label{solv1}\end{equation}
and hence the $H^{-1/r}$ to $L^2$ local solvability of $P$ near $x_0$. This concludes the proof of the theorem.
\end{proof}

\section{Examples of locally solvable mixed-type operators}\label{secExMType}
In this section we will show some examples of operators of mixed-type (\ref{eqP1}) that are locally solvable by virtue of Theorem \ref{thmMixedType}.

\subsection{Example 1.} This is an example of a degenerate Schr\"odinger operator which falls in the mixed-type class, which is $L^2$ to $L^2$ locally solvable.

Let $\Omega_0\subset\mathbb{R}_x^n$ an open set, and consider in $\Omega=\mathbb{R}_t\times\Omega_0$ the Schr\"odinger operator
$$P=f(x)\sum_{j=1}^nD_{x_j}^2+D_t,$$
where $f\in C^\infty(\Omega_0;\mathbb{R})$ is a harmonic function in the $x$-variable
such that $S_0=f^{-1}(0)\not=\emptyset$ and $df\bigl|_{S_0}\not=0.$ 
Therefore the set $S$ in the statement is given here by $S=\mathbb{R}\times S_0$. We therefore think of $f$ as a function of
$(t,x)$ which is constant in the variable $t$. 
Since
$$P=\sum_{j=1}^nD_{x_j}f(x)D_{x_j}+D_t+i\sum_{j=1}^n(\partial_{x_j}f)D_{x_j},$$
we have that $P$ is of the form (\ref{eqP1}) (mixed-type) with $N=n$,
$$X_j=D_{x_j},\,\,1\leq j\leq n,\,\,\,\,X_0=\langle\nabla f(x),D_x\rangle,\,\,\,\,\mathrm{and}\,\,\,\,X_{n+1}=D_t.$$
Since $iX_0f=|\nabla f(x)|^2>0$ on $S$, $\{X_0,X_j\}=-\sum_{k=1}^n(\partial^2 f/\partial x_j\partial x_k)\xi_k,$ $1\leq j\leq n$, $\{X_0,X_{n+1}\}=0$, and
$d_{X_0}=-i\Delta f=0$ by assumption, we have that (HM1), (HM2) and (HM3) are fulfilled. As the characteristic set 
$\Sigma\subset T^*\Omega\setminus 0$ of $\sum_{j=0}^nX_j^*X_j$ is $\{(t,x,\tau,0);\,\,\tau\not=0\}$,
we have that $\pi^{-1}(t_0,x_0)\cap\Sigma\not=\emptyset$ for all $(t_0,x_0)\in\Omega$ and none of conditions (HM4) and (HM5) may hold. Theorem \ref{thmMixedType}(i) thus yields 
that $P$ is $L^2$ to $L^2$ locally solvable near each point of $S$.

\subsection{Example 2.} Consider in $\mathbb{R}^2$ with coordinates $x=(x_1,x_2)$ the functions
$f(x)=x_1-(x_2+x_2^3/3)$ and $g=g(x_2)=1+x_2^2$. For $\alpha>1$ a constant, let
$$X(x,\xi)=g(x_2)\xi_1+\xi_2,\,\,\,\,X_0(x,\xi)=\alpha\xi_1+\frac{\xi_2}{g(x_2)},$$
let
$$X_1(x,\xi)=\sqrt{g(x_2)}\frac{X(x,\xi)}{\sqrt{1+g(x_2)^2}},\,\,\,\,X_2(x,\xi)=\frac{1}{\sqrt{g(x_2)}}\frac{X(x,\xi)}{\sqrt{1+g(x_2)^2}},$$ 
and let
$$X_3(x,\xi)=\mu_1(x)X(x,\xi)+\mu_2(x)X_0(x,\xi),$$
with $\mu_1,\mu_2\in C^\infty$ real valued.
Consider 
the operator
$$P=\sum_{j=1}^2X_j^*fX_j+iX_0+X_3.$$
Since 
$$iX_0f(x)=\alpha-1>0,\,\,\mathrm{and}\,\,\{X_0,X\}(x,\xi)=\frac{\{\xi_2,g\}(x_2)}{g(x_2)^2}X(x,\xi),$$
we have that also $\{X_0,X_3\}$ is a smooth multiple of $X$ and hence that conditions (HM1), (HM2) and (HM3) are fulfilled.
Therefore Theorem \ref{thmMixedType}(i) yields that $P$ is $L^2$ to $L^2$ locally solvable at each point of $S=f^{-1}(0)$.
Note that conditions (HM4) and (HM5) cannot hold in this case.

\subsection{Example 3.} Consider in $\mathbb{R}^3$ with coordinates $x=(x_1,x_2,x_3)$ the operators
$$X_1=D_{x_1},\,X_2=x_1^kD_{x_3},\,X_3=\beta(x)D_{x_1},\,X_0=D_{x_2},$$
where $k\geq 1$ is an integer and $\beta\in C^\infty(\mathbb{R}^3;\mathbb{R})$.
 Let $f(x)=x_2+g(x_1,x_3)$ and let
$$P=\sum_{j=1}^2X_j^*fX_j+iX_0+X_3.$$
It is clear that (HM1) is fulfilled. Since $\{X_0,X_j\}=0,$ $j=1,2,$ and because of the assumption on $\beta$, condition (HM2) is fulfilled,
and by virtue of the fact that 
$d_{X_0}(x)=0,$ we have that also (HM3) is satisfied. As $X_0$, $X_1$ and $X_2$ satisfy the H\"ormander condition
at step $r=k+1\geq 2$ either condition (HM4), when $k=1$, or condition (HM5), when $k\geq 2$, holds
so that Theorem \ref{thmMixedType}(ii) (when $k=1$) or (iii) (when $k\geq 2$) yields that $P$ is $H^{-1/r}$ to $L^2$ locally solvable at
$S=f^{-1}(0)$.  

\subsection{Example 4.} Consider in $\mathbb{R}^3$ with coordinates $x=(x_1,x_2,x_3)$ an open set $\Omega$ intersecting the
plane $x_1=-1$, and the operators $X_j(x,D)$, $0\leq j\leq 3$, with symbols
$$X_0(x,\xi)=\xi_2-x_1\xi_3,\,X_1(x,\xi)=\xi_1-x_3\xi_3,\,\,X_2(x,\xi)=(1+x_1)\xi_3,$$
$$X_3(x,\xi)=\sum_{j=0}^2\Bigl(\beta_j(x)X_j(x,\xi)+\gamma(x)\{X_0,X_j\}(x,\xi)\Bigr),\,\,\beta_1,\beta_2,\gamma\in C^\infty(\Omega;\mathbb{R}).$$ 
We have $d_{X_0}=0$ and
\begin{equation}
\{X_1,X_0\}=-X_2,\,\,\{X_1,X_2\}=(2+x_1)\xi_3,\,\,\{X_2,X_0\}=0.
\label{eqRelationsEx4}\end{equation}
Let $f(x)=x_2+x_2^3/3-x_1x_3.$ Then (HM1) holds. As a consequence of the definition of $X_3$ and of the relations (\ref{eqRelationsEx4}) 
we have that $\{X_0,X_3\}$ is controlled (on the fibers of compact sets of
$\Omega$) by $X_0,$ $X_1$ and $X_2$, whence (HM2) and (HM3) are all satisfied. 
Let $\Omega_\pm:=\Omega\cap\{x_1\gtrless -1\}.$ Note that
since $(x,\xi)\in\Sigma\Rightarrow \xi_3\not=0$ (otherwise we are in the zero-section of $T^*\Omega$), we have
\begin{itemize}
\item[(a)] $\pi^{-1}(\Omega_\pm)\cap\Sigma=\emptyset$,
\item[] while
\item[(b)] if $x_0=(-1,x^0_2,x^0_3)\in\Omega$ then
$$\pi^{-1}(x_0)\cap\Sigma=\{(x_0,\xi)\in T^*\Omega\setminus 0;\,\,\xi_1=x_3^0\xi_3,\,\xi_2=-\xi_3,\,\xi_3\not=0\}\not=\emptyset.$$
\end{itemize}
In case (a) we have that for any given $x_0\in f^{-1}(0)\cap\Omega_\pm$ Theorem \ref{thmMixedType}(iv) yields that $P=\sum_{j=1}^2X_j^*fX_j+X_3+iX_0$ is
$H^{-1}$ to $L^2$ locally solvable near $x_0$.

In case (b), any given $x_0\in f^{-1}(0)\cap\Omega\cap\{x_1=-1\}$  has a fiber which contains characteristic points, and
we may find a (connected) open neighborhood $V\subset\Omega$ of $x_0$ such that
in $\pi^{-1}(V)\cap\Sigma$ the Hamilton fields $H_{X_0},$ $H_{X_1}$ and $H_{X_2}$ are linearly independent and the relations (\ref{eqRelationsEx4})
grant the validity of (HM4) at $x_0$ (and hence for all $\rho\in\Sigma$ with $\pi(\rho)$ belonging to a neighborhood of $x_0$). 
Therefore Theorem \ref{thmMixedType}(ii) yields that $P$ is $H^{-1/2}$ to $L^2$ locally solvable near $x_0$.

\section{The Schr\"odinger-type case}\label{sec3}

Let now $P$ be an operator of the form (\ref{eqP2}), that is,
$$P=\sum_{j=1}^NX_j^*f_jX_j+X_{N+1}+a_0,$$
where, recall, $f_1,\ldots,f_N\in C^\infty(\Omega;\mathbb{R})$. 
Note that the subprincipal symbol of $P$ is given by
$$X_{N+1}(x,\xi)+\sum_{j=1}^N\Biggl(\mathsf{Im}\Bigl((\bar{X}_jf_j)(x)X_j(x,\xi)\Bigr)-\frac{i}{2}f_j(x)\{\bar{X}_j,X_j\}(x,\xi)\hspace{1.5cm}$$
$$\hspace{6.5cm}-f_j(x)\mathsf{Re}\Bigl(\overline{d_{X_j}(x)}X_j(x,\xi)\Bigr)\Biggr).$$

In this section we make the following hypotheses:

\begin{itemize}
\item[(HS1)] $X_1,\ldots,X_N$ have \textit{complex} coefficients;
\item[(HS2)] For all $x_0\in\Omega$ there exists a connected neighborhood $V_{x_0}\subset\Omega$ of $x_0$ and a function $g\in C^\infty(V_{x_0};\mathbb{R})$
such that 
\begin{itemize}
\item[(i)] $X_jg=0$ on $V_{x_0}$ for all $1\leq j\leq N$;
\item[(ii)] $X_{N+1}g\not=0$ on $V_{x_0}$.
\end{itemize}
\end{itemize}

\begin{remark}\label{remthmComplex}
Note that once a function $g$ has been found to satisfy (HS2) (i) and (ii), one may change 
the sign of $g$ so as to have $iX_{N+1}g>0$ on $V_{x_0}$.
\end{remark}

One has the following result.

\begin{theorem}\label{thmComplex}
Let $P$ be of the form \eqref{eqP2} such that conditions (HS1) and (HS2) are satisfied. 
Then for all $x_0\in\Omega$ there exists a compact set $K$ containing $x_0$ in its interior 
such that the operator $P$ is $L^2$ to $L^2$ locally solvable in $\mathring{K}$ (the interior of $K$).
\end{theorem}

\begin{proof}
We have to obtain an $L^2$ a priori estimates for the adjoint $P^*$, which may be written as
$$P^*=\sum_{j=1}^NX_j^*f_jX_j+X_{N+1}^*+\bar{a}_0,$$
where it is important to note that since $X_{N+1}(x,\xi)=\overline{X_{N+1}(x,\xi)}$, then $X_{N+1}(x,D)^*=X_{N+1}(x,D)+d_{X_{N+1}}(x)$.
Let now $x_0\in S$ and let $K\subset V_{x_0}$ be a compact set containing $x_0$ in its interior $\mathring{K}=:U$. Let $g$ be a function such that (HS2-i) and (HS2-ii) hold
with $iX_{N+1}g>0$ on $K$.
For $\varphi\in C_0^\infty(K)$ and $\lambda>0$  to be picked later on we consider
\begin{equation}
\label{CP*}
\mathsf{Im}\,(e^{\lambda g}P^*\varphi,e^{\lambda g}\varphi)=\sum_{j=1}^N\underset{(\ref{CP*}.1)}{\mathsf{Im}\,(e^{\lambda g}X_j^*f_jX_j\varphi, e^{\lambda g}\varphi)}
\end{equation}
$$\hspace{5cm}+\underset{(\ref{CP*}.2)}{\mathsf{Im}\,(e^{\lambda g}X_{N+1}^*\varphi,e^{\lambda g}\varphi)}+\underset{(\ref{CP*}.3)}{\mathsf{Im}\,(e^{\lambda g}\bar{a}_0\varphi,e^{\lambda g}\varphi)}.$$
We separately estimate the three terms $(\ref{CP*}.1)$, $(\ref{CP*}.2)$ and $(\ref{CP*}.3)$ in \eqref{CP*}. 

As regards (\ref{CP*}.1), for all $1\leq j\leq N$ and for all $\varphi\in C_0^\infty(K)$ we have
\begin{equation}
\mathsf{Im}\,(e^{\lambda g}X_j^*f_jX_j\varphi, e^{\lambda g}\varphi)=\mathsf{Im}\,(f_jX_j\varphi, X_j(e^{2\lambda g}\varphi))
\label{term1}\end{equation}
$$=\mathsf{Im}\,(f_jX_j\varphi, 2\lambda (X_jg)e^{2\lambda g}\varphi)+\mathsf{Im}\,(f_jX_j\varphi, e^{2\lambda g}X_j\varphi)=0,$$
because $X_jg=0$ and $(f_jX_j\varphi, e^{2\lambda g}X_j\varphi)\in\mathbb{R}$.

As regards (\ref{CP*}.2), for all $\varphi\in C_0^\infty(K)$ we have
$$\mathsf{Im}\,(e^{\lambda g}X_{N+1}^*\varphi,e^{\lambda g}\varphi)=\mathsf{Im}\,(\varphi,X_{N+1}(e^{2\lambda g}\varphi))$$
$$=\mathsf{Im}\,(\varphi,2\lambda(X_{N+1}g)e^{2\lambda g}\varphi)+\mathsf{Im}\,(\varphi,e^{2\lambda g}X_{N+1}\varphi)$$
$$=\mathsf{Im}\,(\varphi,2\lambda(X_{N+1}g)e^{2\lambda g}\varphi)+\mathsf{Im}\,(\varphi,e^{2\lambda g}X_{N+1}^*\varphi)-\mathsf{Im}\,(\varphi,d_{X_{N+1}}e^{2\lambda g}\varphi).$$
Therefore 
$$\mathsf{Im}\,(e^{\lambda g}X_{N+1}^*\varphi,e^{\lambda g}\varphi)=\frac12\Bigl[\mathsf{Im}\,i(\varphi,2\lambda(iX_{N+1}g)e^{2\lambda g}\varphi)
+\mathsf{Im}\,(d_{X_{N+1}}e^{2\lambda g}\varphi,\varphi)\Bigr].$$
Since $iX_{N+1}g>0$ near $x_0$, there exists a compact set $K_0\subset V_{x_0}$ containing $x_0$ in its interior and a positive constant $c_0$ such that $iX_{N+1}f\geq c_0$ on $K_0$.
We can then shrink the compact set $K$ around $x_0$ to a compact contained in $K_0$, that we keep denoting by $K$, in such a way that for $\lambda >0$
and for all $\varphi\in C_0^\infty (K)$ we have
$$\mathsf{Im}\,i(\varphi,2\lambda(iX_{N+1}g)e^{2\lambda g}\varphi)\geq 2\lambda c_0|\!|e^{\lambda g}\varphi|\!|_0^2,$$
and thus
\begin{equation}
\mathsf{Im}\,(e^{\lambda g}X_{N+1}^*\varphi,e^{\lambda g}\varphi)\geq  \lambda c_0 |\!|e^{\lambda g}\varphi|\!|_0^2-\frac12|\!|d_{X_{N+1}}|\!|_{L^\infty(K)}|\!|e^{\lambda g}\varphi|\!|_0^2.
\label{term2}\end{equation}

As for the term (\ref{CP*}.3), we have for all $\varphi\in C_0^\infty (K)$
\begin{equation}
\mathsf{Im}\,(e^{\lambda g}\bar{a}_0\varphi,e^{\lambda g}\varphi)\geq -|\!|a_0|\!|_{L^\infty(K)}|\!| e^{\lambda g}\varphi|\!|_0^2,
\label{term3}\end{equation}
whence, by inserting \eqref{term1}, \eqref{term2} and \eqref{term3} into \eqref{CP*}, we find that for all
$\varphi\in C_0^\infty(K)$ and all $\lambda>0$ 
\begin{equation}
\mathsf{Im}\,(e^{\lambda g}P^*\varphi,e^{\lambda g}\varphi)\geq \Big(\lambda c_0- \frac{|\!| d_{X_{N+1}}|\!|_{L^\infty(K)}}{2}-|\!|a_0|\!|_{L^\infty(K)}\Big) |\!| e^{\lambda g}\varphi|\!|_0^2.
\label{final}\end{equation}
Fixing $\lambda>0$ sufficiently large yields the existence of $C>0$ such that
$$|\!|P^*\varphi|\!|_0^2\geq C|\!|\varphi|\!|_0^2,\,\,\,\,\forall\varphi\in C_0^\infty(K),$$
and concludes the proof of the theorem.
\end{proof}

\begin{remark}\label{remSchroedinger}
Suppose $B\colon C_0^\infty(V_{x_0})\longrightarrow C_0^\infty(V_{x_0})$ is a zeroth order properly supported pseudodifferential operator such that $B^*=B+R,$ where $R$ is a smoothing operator.
One then has 
$$\mathsf{Im}(P^*\varphi,B\varphi)=\sum_{j=1}^N\mathsf{Im}(X_j\varphi,f_j[X_j,B]\varphi)+\frac12\sum_{j=1}^N\mathsf{Im}(X_j\varphi,[f_j,B]X_j\varphi)$$
$$\hspace{2.5cm}+\mathsf{Im}(\varphi,[X_{N+1},B]\varphi)+O(|\!|\varphi|\!|_0^2),$$
where in $O(|\!|\varphi|\!|_0^2)$ we have the contributions of $[R,X_j]\varphi$, $[R,X_{N+1}]\varphi$ and $[B,d_{N+1}]\varphi$. The first two terms to the right give problems, for one
is not able to control norms of the kind $|\!|X_j\varphi|\!|_0$, the only usable term being given by the third one. 
This suggests that, in this setting, to be able to exploit condition (HS2-ii)
a resonable choice of $B$ is indeed $B=e^{\lambda g}$. 
\end{remark}

\section{Examples of locally solvable Schr\"odinger-type operators}\label{secExSType}

In this section we exhibit some examples to which Theorem \ref{thmComplex} can be applied to conclude $L^2$ to $L^2$ local solvability.

\subsection{Example 1.} In $\mathbb{R}_t\times\mathbb{R}_x^n\times\mathbb{R}_y^m$ we consider the operators
$$P_1=-\Delta_x-\Delta_y+D_t,\,\,\,\,P_2=-\Delta_x+\Delta_y+D_t,\,\,\,\,P_3=f_1(t)\Delta_x+f_2(t)\Delta_y+D_t,$$
where $f_1,$ $f_2$ are smooth, non-identically zero functions of $t$ only. Then $P_1,$ $P_2$ and $P_3$ are all $L^2$ to $L^2$ locally solvable.

\subsection{Example 2.} This example is related to the so-called Mizohata structures (see \cite{T1} or \cite{BCH}). 
Let $\Omega_0\subset\mathbb{R}^n_x\times\mathbb{R}_y$ be an open set and
consider in $\mathbb{R}_t\times\mathbb{R}^n_x\times\mathbb{R}_y$ the open set 
$\Omega=\mathbb{R}_t\times\Omega_0.$ Let $Q=Q(x)$ be a real-valued quadratic form and let
$$X_j=D_{x_j}-i\frac{\partial Q}{\partial x_j}(x)D_y,\,\,\,\,1\leq j\leq n.$$
Let $Y=Y(x,y,D_x,D_y)$ be a first order homogeneous differential operator with real symbol and finally let
$$X_{N+1}=D_t+Y.$$
Then the function $g=g(t)=t$ satisfies the assumptions (HS1) and (HS2) and the operator 
$$P=\sum_{j=1}^nX_j^*f_jX_j+X_{N+1}+a_0$$
is $L^2$ to $L^2$ locally solvable near each point of $\Omega$, whatever the choice of the (non-identically zero)
$f_j\in C^\infty(\Omega;\mathbb{R})$ (and of $a_0\in C^\infty(\Omega;\mathbb{C})$).

\subsection{Example 3.}
In $\mathbb{R}^4$ with coordinates $x=(x_1,x_2,x_3,x_4)$ let $\Omega\subset\mathbb{R}^4$ be open and let
$$X_1=D_1-i\frac{x_2}{2}D_3,\,\,\,\,X_2=D_2+i\frac{x_1}{2}D_3,\,\,\,\,X_3=D_4+\alpha(x)D_3,$$
where $\alpha\in C^\infty(\Omega;\mathbb{R}).$
Then, choosing $g=g(x_4)=x_4$ we have that, whatever the (non-identically zero) functions $f_1,f_2\in C^\infty(\Omega;\mathbb{R})$ 
(and of $a_0\in C^\infty(\Omega;\mathbb{C})$), the operator
$$P=X_1^*f_1X_1+X_2^*f_2X_2+X_3+a_0$$
is $L^2$ to $L^2$ locally solvable near each point of $\Omega$.  

\begin{remark}
The point in the Examples 2 and 3 above is to work with a ``cylindric'' geometry, in which a system of \textit{complex} vector fields $\boldsymbol{X}=\{iX_1,\ldots,iX_N\}$ 
is given to be locally tangent (in the sense that the real parts and the imaginary parts of the vector fields are tangent) to the level sets $L_c=g^{-1}(c)$ of some smooth real-valued function $g$, 
the \textit{real} vector field $iX_{N+1}$ being transverse to the $L_c$ (for $c$ near some regular value $c_0$ of $g$). 
One may very well choose the system $\boldsymbol{X}$ to be a locally involutive system or, more specifically, 
spanning a hypo-analityc structure in the sense of \cite{T1} on each level set $L_c$, with at least one real first-integral. 
Keeping the vector field $iX_{N+1}$ transverse to the $L_c$, one may then think of $P$ as an evolution operator associated with
the involutive/hypo-analytic structure on the leaves $L_c$ in the direction $iX_{N+1}$.
\end{remark}

\begin{remark}
The operators considered in Sections \ref{sec3} and \ref{secExSType} resemble very much the 
Schr\"odinger operator $D_t+\Delta_x$. In studying them one gives up all possible 
extra information coming from lower order terms, that might interfere with the term $D_t$. This explains, to some extent, the local
$L^2$ existence result.
\end{remark}

\vspace{.5cm}


\end{document}